\documentclass{amsproc}

\usepackage{amssymb}

\usepackage{graphicx}

\usepackage{}

\newtheorem{theorem}{Theorem}[section]
\newtheorem{lemma}[theorem]{Lemma}
\newtheorem{proposition}[theorem]{Proposition}
\newtheorem{corollary}[theorem]{Corollary}

\theoremstyle{definition}

\newtheorem{example}[theorem]{Example}

\theoremstyle{remark}
\newtheorem{remark}[theorem]{Remark}

\numberwithin{equation}{section}

\def\mbi#1{\boldsymbol{#1}} %math bold italic. 添え字で大きさが変わる．
 
\def\R{{\boldsymbol{R}}}

\begin{document}

\title{Tropical geometry of PERT}


%    author one information
\author{Masanori Kobayashi}
\address{Department of Mathematics and Information Sciences, Tokyo Metropolitan University, 1-1 Minami-Ohsawa, Hachioji-shi, Tokyo, 192-0397, Japan}
\curraddr{}
\email{kobayashi-masanori@tmu.ac.jp}
\thanks{The first author is supported by Grant-in-Aid for Scientific Research 21540045}

%    author two information
\author{Shinsuke Odagiri}
\address{Department of Mathematics and Information Sciences, Tokyo Metropolitan University, 1-1 Minami-Ohsawa, Hachioji-shi, Tokyo, 192-0397, Japan}
\curraddr{}
\email{odagiris@tmu.ac.jp}
\thanks{The second author thanks the support by the MEXT program ``Support Program for Improving Graduate School Education"}

%\subjclass[2000]{Primary }
%    The 2010 edition of the Mathematics Subject Classification is
%    now available.  If you are citing a classification from the
%    new scheme, use the following input coding instead.

\date{February 29, 2012}


\begin{abstract}
Based on a description of project networks by max-plus algebra and poset, 
the adjacency of critical paths is presented using tropical geometry. 
\end{abstract}

\maketitle

\section{Introduction}
The max-plus algebra, also known as the tropical semiring, appears in discrete event systems\cite{BCOQ}. 
For instance, the event firing time of a Petri net is analysed by max-plus linear algebra. 
PERT (Program Evaluation and Review Technique) and CPM (Critical Path Method) are popular methods of scheduling, and the problem of resource conflict is discussed in \cite{KTG}. 

Suppose there is a set of several activities with an order; 
each activity can be started after all the preceding activities have accomplished. 
The activities form a {\it project network\/}.
Each activity $x$ is assigned with a nonnegative number $t_x$ representing the time required for $x$, which is called the {\it time cost} of $x$. 
One usually includes the start activity $u$ and the end activity $y$ both with time cost zero.

Fix a project network. 
A successive series of activities from $u$ to $y$ is called a {\it path\/}. 
For the accomplishment of a path $I$, we need at least the sum of the time costs along $I$. 
%The sum of the time costs along a path $I$ is the minimum time required for the accomplishment of $I$ only. 
Take the maximum 
%of the sums 
for all the paths, and you get the {\it earliest finishing time} for the project itself. 
The path attaining the maximum is called a {\it critical path}. 

Since the critical path determines the time necessary to complete all the activities, 
it is important to watch and control the critical path. 
If one can reduce the time cost of an activity on the critical path, 
the total duration becomes shorter. 
%CPM seek a way to efficient cost cut on the path. 
On the contrary, an accident on the critical path directly results in the delay of the closing. 

%If you reduce the time cost too much, or 
However, if an activity out of the critical path suddenly requires more time, 
or some activities in the critical path reduces their time costs,
the critical path might change. 
%is it so much as to change the critical path? 
If so, which path is likely to be critical?
In this paper, the question is answered geometrically by the adjacency of paths. 
%and quantitatively by the tropical distance to walls. 

We formulate a project network as a graph and an ordered set in Section 2, 
and show when a tropical polynomial is realised as the earliest finishing time in Section 3. 
The topology of the transition of critical paths is discussed in Sections 4 and 5. 

Acknowledgment: 
We thank to St\'ephane Gaubert and Jean-Jacques Risler for discussion and kindly noticing us good references. 

\section{Graph representation and poset structure of a project network}
A project network is usually represented by a directed graph called a {\it PERT chart\/}, 
and there are two popular ways; 
namely, {\it activity-on-arrow} (AOA) and {\it activity-on-node} (AON) diagrams. 
%AON:activity network, AOA:arrow diagram
In the AOA diagram, as imagined literally,
each activity is represented by an arrow (a directed edge). 
A vertex represents a milestone between activities. 
On the other hand, in the AON diagram, 
an activity is represented by a vertex and a dependency between activities by an arrow. 
Both diagrams can represent any project networks, 
but we adopt the AON diagram, 
where the insertion of dummy arrows (e.g. the dotted arrow in Figure 1, with no corresponding activities)
%, which is sometimes necessary for the AOA diagram, 
 are not needed. 
%We omit $u$ and $y$ unless necessary. 
\begin{figure}[htbp]
\includegraphics{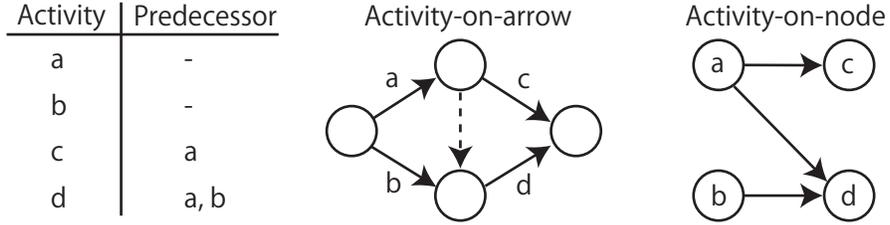}
\caption{PERT charts ($u$ and $y$ are omitted)}
\end{figure}

Every AON diagram is {\it simple}, that is, has no self-loops or multiple arrows. 
Thus an arrow is represented by a pair of two different (namely, initial and terminal) vertices. 
Logically there are no circular references, so the graph contains no cycles. 
%Logically there are no repetition of an activity, so the graph contains no cycles. 
By the semipositivity of time costs we may assume 
that an arrow connects only a pair of predecessor-successor (i.e. there are no activities in between); e.g. if $x_0 \to x_1 \to \cdots \to x_k$ ($k \geq 2$) exists then we do not include the arrow $x_0 \to x_k$. 
We call that final condition as having no `{\it short-cuts}'. 
%The last condition is for the sake of simplicity and not essential. 
%or rather include all the dependencies! 

We shall identify such a graph and a partially ordered set (`poset') as follows. 
$X$ is supposed to be the set of activities. 

\begin{proposition}\label{poset-PERT_prop}
Let $X$ be a finite set. 
There is a one-to-one correspondence between the set of partial orders on $X$ and the set of simple directed graphs with vertex set $X$ without cycles or short-cuts. 
\end{proposition}
\begin{proof}
First note that for a finite ordered set, $x < x'$ is equivalent to the existence of a finite sequence of predecessor-successors $x=x_0 < \cdots < x_k=x'$. 
This can be done by inductive insertions of intermediate elements if exist. 

Assume a partial order is given to $X$. 
Combining a predecessor and a successor by an arrow, 
one get a simple directed graph without short-cuts. 
The antisymmetric law guarantees the vacancy of cycles. 
The graph is nothing but a Hasse diagram. 

On the contrary, from a given directed graph, 
the reachability along a finite ($ \geq 0$) sequence of arrows defines a partial order. 
The reflectivity and transitivity is obvious. 
$x < x'$ and $x' < x$ does not occur simultaneously since there are no cycles. 

It is easy to see that the compositions of those two maps in both directions are identity maps. 
\end{proof}

Thus, a PERT chart without short-cuts is equivalent to a Hasse diagram of a poset with a minimum element $u$ and a maximal element $y$ such that the vertices are nonnegatively weighted. 
A path is a graph-theoretic path from $u$ to $y$ in the Hasse diagram, 
which consists of a totally-ordered set of vertices. 
Assumption of `no short-cuts' guarantees that the ordered set is maximal; 
if one can add another activity, the insertion must be somewhere in-between, thus a short-cut exists. 
On the contrary, 
suppose a poset $X$ with $u$ and $y$ is given. 
%By Proposition \ref{poset-PERT_prop}, there exists a corresponding Hasse diagram of $X$. 
Take a maximal totally-ordered subset $I$. 
Since $I$ is maximal, $I$ contains $u$ and $y$, 
and the predecessor-successor relation on $I$ extends to that on $X$. 
Thus $I$ corresponds to a path of the PERT chart.

We omit $u$ and $y$ from $X$ unless necessary from now on. 
%$u$ and $y$ are always the minimum and the maximum element, 
%so we omit them from $X$ unless necessary from now on. 

\begin{lemma}\label{embedding_lemma}
Every finite poset $(X,\leq)$ can be embedded into a totally-ordered set $(X, \leq')$ with the same set. 
\end{lemma}
\begin{proof}
Add a minimum element $u$ to $X$. 
Order the elements of $X$ by the order distance from $u$ with respect to $\leq$, and define $\leq'$ by being closer to $u$ is smaller. 
The order for the elements with same distance is arbitrary. 
If $x \leq x'$, then the distance to $x'$ is greater than or equal to that to $x$.
Thus the identity map is monotone. 
\end{proof}

By the lemma, 
%for a nonnegative integer $n$, 
we may always assume that a finite poset of order $n$ is equivalent to $[n]:=\{1,2,\ldots,n\}$ as a set, 
where $i < j$ in the poset order implies $i < j$ as integers. 
The other implication is not necessarily true. 
A totally-ordered set corresponds to a serial PERT chart, 
whereas in the case of a parallel PERT chart, as a poset $[n]$ is an {\it antichain\/}, where distinct elements are not comparable. 

%\begin{remark}
%In real PERT charts, especially AOA diagrams, 
%the vertices are sometimes named as $10, 20, 30, \ldots$ in time order. 
%This is because when one want to insert a new vertex later, 
%say between 10 and 20, one can use `15'. 
%\end{remark}

\section{Earliest finishing time as a tropical polynomial}

Let $F$ be the earliest finishing time of a project network 
with $n$ activities. 
%with activities $1,2,\ldots,n$. 
We denote the time costs by $t_1,\ldots,t_n$. 
Since the change of states matters, the time costs are treated as variables, not as constants. 

For a subset $I$ of $[n]$, we write the monomial term $\prod_{i \in I} t_i$ as $t_I$. 

\begin{proposition}\label{prerealisable}
The earliest finishing time $F$ can be written as a tropical polynomial of $t_1,\ldots,t_n$ satisfying the following three conditions:
%For a PERT chart with vertex set of order $n$, 
%the earliest finishing time $F$ is a tropical polynomial of $n$ variables satisfying the following three conditions:
\begin{enumerate}
 \item the degree of $F$ on each variable is at most one,
 \item the coefficient of each term is a unit,
% \item every variable appears in $F$, 
 \item no term is divisible by any other terms. (`nondivisibility')
\end{enumerate}
\end{proposition}
\begin{proof}
For each path, the sum of the time costs is a tropical monomial, 
and the maximum is represented by the tropical addition. 
(1) follows since every path contains each activity at most once. 
(2) is obvious. 
Thus every term of $F$ can be written as $t_I$ for some $I$,
where $I$ is a maximal totally-ordered subset of $[n]$. 
% by Proposition \ref{poset-PERT_prop}.
Suppose a term $t_J$ is divisible by a term $t_I$. 
That condition is equivalent to that $I \subset J$ holds. 
Since $I$ is maximal, $I=J$ follows. 
%Take $j \in J \setminus I$. 
%Since $j$ is comparable to every element of $J$, 
%it is comparable to every element in $I$. 
%If $j$ is the minimum (resp. maximal) element of the totally ordered set $I \cup \{ j \}$,
%then the starting node (resp. ending node) of the path corresponding to $t_I$ is actually not a starting node
%(resp. ending node) of the path corresponding to $t_J$, which cannot happen as a PERT chart.
%Thus suppose there exists a predecessor $i_1$
%and a successor $i_2$ of $j$ of the totally ordered finite set $[n] \cup \{ j \}$.
%Then the arrow $i_1 \rightarrow i_2$ is a short-cut of the path of $J$, 
%which again contradicts to the condition of a PERT chart.
%Thus the dividend $t_J$ of $t_I$ does not exist.
\end{proof}

Let us think of the inverse problem: what is a sufficient condition for
a tropical polynomial to have a corresponding PERT chart? 
%We know that all the coefficients must be units and are of degree 1 for any variables.

We call a nonconstant tropical polynomial $F$ {\it prerealisable} if
$F$ satisfies the three conditions of Proposition \ref{prerealisable}. 
Prerealisable polynomial can be written in the form of $\sum_{I \in {\mathcal I}} t_I$, where ${\mathcal I}$ is a nonempety subset of the power set $2^{[n]}$ 
which satisfies that if $I,J \in \mathcal I$ and $I \subset J$ then $I=J$. 
Moreover, 
We say that a tropical polynomial is {\it PERT-realisable} (or simply, {\it realisable}) if there exists a PERT chart with the earliest
finishing time being $F$. 
By Proposition \ref{prerealisable}, a realisable tropical polynomial is prerealisable.

\begin{proposition}
Let $F=\sum_{I \in {\mathcal I}} t_I$ be a tropical polynomial of $n$ variables.
Then $F$ is realisable iff there exists
%Then there exists a PERT chart with the earliest finishing time being $F$ iff there exists
a poset structure on the index set $[n]$ such that 
$$ I \ \text{is a maximal totally-ordered subset} \Leftrightarrow t_I \ \text{is a term of} \ F.$$
\end{proposition}
\begin{proof}
If $F$ is realisable, then we have the equivalence %of the conditions
since they both mean that $I$ is a path of the PERT chart.
Suppose that the index set has a poset structure satisfying the given equivalence. 
Then from Proposition \ref{poset-PERT_prop}, 
we have a corresponding PERT chart such that every maximal totally-ordered subset
corresponds to a path and vice versa. 
Thus $F$ is the earliest finishing time of the PERT chart. 
%On the other hand, if the index set has a poset structure satisfying the given equivalence,
%then the Hasse diagram gives the corresponding PERT chart.
\end{proof}

Prerealisable polynomials are not always realisable as shown below. 

\begin{proposition}
Let $\mathcal I$ be a subset of $2^{[n]} \setminus \{[n]\}$ such that for every $i,j \in [n]$,
there exists $I \in \mathcal I$ including both $i$ and $j$. Then $F=\sum_{I \in \mathcal I} t_I$ is not realisable.
\end{proposition}
\begin{proof}
Suppose $F$ is realisable. 
Then $[n]$ has a poset structure $\leq$ such that
every $I \in \mathcal I$ is totally ordered.
Then $([n],\leq)$ is totally ordered 
%Then the two poset structures $\leq$ and $\leq'$ must coincide 
since every two elements are comparable. 
This contradicts to the maximality of $I$.
\end{proof}

Let $n,k$ be positive integers satisfying $k \le n$. 
Take $\mathcal I$ to be $\{I \subset [n] \ | \ \# I =k \}$. 
We write the corresponding tropical polynomial $F$ as $F_{n,k}$. 
Apparently $F_{n,k}$ is prerealisable. 

\begin{corollary}\label{realisable}
$F_{n,k}$ is realisable iff $k=1$ or $n$.
\end{corollary}
\begin{proof}
If $k=1$ (resp. $k=n$), then $F_{n,k}$ corresponds to the PERT chart consisting of $n$ parallel (resp. serial) activities. 
On the other hand, if $2 \le k \le n-1$, $F_{n,k}$ is not realisable by the previous proposition.
\end{proof}

A simple nonrealisable example is $F_{3,2}=t_1t_2 \oplus t_2t_3 \oplus t_1t_3$. 

\section{Tropical hypersurface and critical paths}

We consider the topology of the space of paths. 
First we fix the notation. 

%Let $\mbi{T}$ be the tropical semifield $\{-\infty\} \cup \R$. 
We regard $\mbi{t}=(t_1,t_2,\ldots,t_n)$ as the tropical coordinates of $\mbi{R}^n$, 
though we mainly work on the semipositive orthant $\R_{\geq 0}^n$, which we denote by $\varGamma$. 
%The value of the monomial term $t_I$ at a point $p \in \T^n$ is denoted by $p_I$.
%This $p_I$ equals $\Sigma_{i \in I} p_i$,
%where the addition is usual (i.e. non-tropical) and $p_i$ is the $i$-th coordinate of $p$.
For $I \subset [n]$, we denote by $\mbi{e}_I$ the element of $\varGamma$ such that 
$t_i$ equals one for $i \in I$ and zero otherwise. 
%\[
%  t_i = \begin{cases}
%    1 & (i \in I), \\
%    0 & (\text{otherwise}).
%  \end{cases}
%\]
When $I=\{i\}$, we write $\mbi{e}_{\{i\}}$ as $\mbi{e}_i$. 
%We may write $\mbi{e}_{\{i_1,\ldots,i_k\}}$ as $\mbi{e}_{i_1,\ldots,i_k}$. 

For a tropical polynomial $F=F(t_1,\ldots,t_n)$, 
the tropical hypersurface $V(F)$ in $\R^n$ is the locus where more than one terms attain the maximum. 
We denote by $V(F)_{\geq0}$ the intersection $V(F) \cap \varGamma$，
and by $\varGamma_F$ the complement $\varGamma \smallsetminus V(F)_{\geq 0}$.

%Take a tropical polynomial 
%$F=\sum_{I \in \mathcal I} t_I$ for $\mathcal I \subset 2^{[n]}$. 
%$\varGamma_F$ is decomposed into the sum of $C^\circ_I:=\bigcap_{J \in \mathcal I, J \neq I} \{ \mbi{t} \in \varGamma \ | \ t_I > t_J\}$ for $I \in \mathcal I$, where $t_I$ solely takes the maximum value among the monomials in $F$. 

In the sequel, fix a tropical polynomial 
$F=\sum_{I \in \mathcal I} t_I$ and $I \in \mathcal I$. 
We define a close (resp. an open) convex polyhedral cone $C_I$ (resp. $C_I^\circ$) to be
%a subset $C_I$ (resp. $C_I^\circ$) to be 
$\bigcap_{J \in \mathcal I} \{ \mbi{t} \in \varGamma \ | \ t_I \ge t_J\}$ (resp. 
$\bigcap_{J \in \mathcal I} \{ \mbi{t} \in \varGamma \ | \ t_I > t_J\}$). 
%Then $C_I$ is a closed convex polyhedral cone.
Since $V(F)_{\geq 0}=\bigcup_{\substack{J,K\\ J \neq K}} (C_J \cap C_K)$, 
$V(F)_{\geq 0}$ is also a cone. 
% over the apex $\mbi{0}=(0,\dots,0)$.

\begin{proposition}
If $F$ is prerealisable, the following holds:
\begin{enumerate}
 \item $C^\circ_I$ is a nonempty and $C_I$ is $n$-dimensional.
%open convex cone in $\varGamma$. 
% \item $C_I$ is an $n$-dimensional closed convex polyhedral cone. 
 \item $C^\circ_I$ is the interior of $C_I$ in $\varGamma$. 
 \item $C_I$ is the closure of $C^\circ_I$ in $\varGamma$. 
\end{enumerate}
\end{proposition}

\begin{proof}
We first prove $(1)$. 
The value of $t_I$ at $\mbi{e}_I$ equals $\# I$, and the value of a term $t_J$ at $\mbi{e}_I$
is $\# (I \cap J)$. Thus $t_I$ is strictly the biggest term of $F$ at $\mbi{e}_I$ by nondivisibility, 
thus $\mbi{e}_I \in C^\circ_I$. 
%As $C^\circ_I$ (resp. $C_I$) is an intersection of homogeneous open (resp. closed) half-spaces, 
%it is an open (resp. a closed) convex cone. 
$C_I$ contains an nonempty open subset $C^\circ_I$, thus is $n$-dimensional. 

We now prove $(2)$. 
Suppose $\mbi{t} \in C_I$ satisfies $t_I=t_J$ for some distinct $I,J$ in $\mathcal I$. 
By nondivisibility, there exists $j \in J \smallsetminus I$. 
Then for any $\varepsilon >0$, 
$\mbi{t}+\varepsilon \mbi{e}_j$ is in $\varGamma$ but not in $C_I$ since $t_I < t_J$. 
This shows that $\mbi{t}$ is a boundary point of $C_I$ in $\varGamma$, 
and the interior of $C_I$ in $\varGamma$ is contained in $C^\circ_I$. 
$C^\circ_I$ is open in $\varGamma$, thus is the interior of $C_I$. 

For (3), 
let $\mbi{t}$ be a point in $C_I$. 
For any $\varepsilon >0$, 
$\mbi{t}+\varepsilon \mbi{e}_I$ is in $\varGamma$ and belongs to $C^\circ_I$ by nondivisibility. 
Thus $\mbi{t}$ is in the closure of $C^\circ_I$. 
Since $C_I$ is closed, it coincides to the closure of $C^\circ_I$. 
\end{proof}

A prerealisable polynomial has a remarkable characterization as a function on the semipositive orthant.

\begin{proposition}\label{Gamma_divided}
A tropical polynomial $F=\sum_{I \in \mathcal I} t_I$ satisfies the nondivisibility condition iff 
each $t_I$ solely attains the maximum at some point in $\varGamma$. 
%connected component $C^\circ$ of $\varGamma_F$. 
%Each term $t_I$ in $F$ attains the maximum alone on the nonempty connected open set $C^\circ_I$. 
\end{proposition}

\begin{proof}
If $F$ satisfies the nondivisibility condition, 
then $C_I^\circ$ is nonempty for all $I$ from the previous proposition.
On the contrary, suppose $I \subset J$, where $I,J \in \mathcal I$. 
Then $\{ \mbi{t} \in \varGamma \ | \ t_I > t_J\}$ is an empty set, 
and we have $C_I^\circ=\emptyset$. 
\end{proof}

\begin{theorem}\label{term_and_connected_component}
Let $F$ be the earliest finishing time of a project network. 
Then $V(F)_{\geq0}$ is the set of points in $\varGamma$ which has two or more critical paths. 
The points in each connected component of $\varGamma_F$ has the same unique critical path. 
Every path becomes a critical path for some time costs. 
\end{theorem}
\begin{proof}
By the definition of critical path and tropical hypersurface, 
$\mbi{t} \in V(F)$ is equivalent to that there are more than one critical paths for the time cost $\mbi{t} \in \varGamma$. 

We show by prerealisablity that $C^\circ_I$ is a connected component of $\varGamma_F$ and $\varGamma_F=\bigcup_{I \in  \mathcal I} C^\circ_I$. 
%The difference of any two tropical monomials in $F$ is a nonzero linear function, thus has a same sign on each connected component $U$ of $\varGamma_F$. 
Tropical monomials are linear functions, hence continuous. 
Thus, the difference of any two tropical monomials in $F$ has a same sign on each connected component $U$ of $\varGamma_F$. 
Hence the maximum term is the same on $U$ and $U$ is contained in some $C^\circ_I$. 
%This shows that the maximum term is the same for the points in $U$. 
%Any two points with the common sole maximum term $t_I$ are connected by a segment in $C^\circ_I$ since $C^\circ_I$ is convex. 
Since $C^\circ_I$ is convex, 
if at $\mbi{t}$ and $\mbi{t}'$ 
the tropical polynomial $F$ attains a same sole maximum term $t_I$, 
the points are connected by a segment in $C^{\circ}_I$. 

The last statement follows from the nonemptiness of $C^\circ_I$. 
\end{proof}
%$V(F)_{\geq 0}$ is not $n$-dimensional. 

\begin{remark}
If we allow the time costs to take negative values, the previous theorem still remains valid. 
The general version may be useful for possible generalizations like this. 
In the problem of percolation or invasion, 
the AND condition `all the previous activities is finished' in PERT is replaced to the OR condition `at least one previous activity is finished'. 
In this case we have the MIN-plus algebra instead. 
By multiplying $(-1)$ to all the time costs, 
one can regard the problem as a negative-time-costs version of PERT. 
%One should be careful that short-cuts are allowed, but still nondivisibility holds since one discard the longer paths.
\end{remark}

\section{Adjacency of paths}
When the time costs change, on purpose or by accident, 
a change of critical paths may follows. 
Let us define a graph to monitor the transition. 
%Let us define a graph through which we can monitor the transition. 
%the changes of the critical paths. 

Take two maximal paths $I,J$ in a PERT chart with the set of paths $\mathcal I$, 
or in general, two monomial terms $t_I,t_J$ in a prerealisable polynomial $F$. 
We say $C_I,C_J$ are {\it adjacent} (or $I,J$ are {\it adjacent}), 
if $C_I,C_J$ have a common codimension-one wall. 
When a critical path changes, some of the adjacent paths always attain the maximum (at least in a moment). 

Let $G(F)$ be a non-oriented graph with the vertex set $\mathcal I$ and the edges between adjacent vertices. 
%Define $G(F)$ to be 
%%the non-oriented adjacency graph of $\mathcal I$. 
%a non-oriented graph given by the followings:
%\begin{itemize}
% \item the vertex set is $\mathcal I$,
% \item an edge connects a pair of adjacent vertices. 
%\end{itemize}
We call $G(F)$ the {\it adjacency graph} of $F$.
By Proposition \ref{Gamma_divided}, we identify each vertex $I$ with the chamber $C_I$. 

%Generally, we can monitor the changes of
%the critical paths through this graph.

\begin{example}\label{complete}
$F_{n,1}=t_1 \oplus \dots \oplus t_n$ in Corollary \ref{realisable} is realisable as a parallel chart. 
For every $i,j\in [n]$, $C_{\{i\}}$ and $C_{\{j\}}$
are adjacent at neighbourhood of a point $\mbi{e}_{\{i,j\}}$ along the codimension-one wall $\{ t_i=t_j \}$.
Thus we have $G(F_{n,1})=K_n$, the complete graph with $n$ vertices.

For $2 \le k \le n-1$, chambers $C_I, C_J$ of $F_{n,k}$ are adjacent iff $\#(I \cap J)=k-1$, 
which is because $C_I=\bigcap_{\substack{i \in I, \, j \in [n] \setminus I}} \{\mbi{t} \in \varGamma \ | \ t_i \ge t_j\}$.
% $C_I=\{t_i \ge t_j \ | \ i \in I, \ j \in [n] \setminus I \}$.
\end{example}

\begin{example}
%In many cases, $G(F)$ is a complete graph. 
Let $F=\sum_{i=1}^{n-1} t_it_{i+1} \ (n \ge 2)$. 
%Then $G(F)$ is $K_{n-1}$. 
Then $G(F)$ is $K_{n-1}$, again a complete graph. 
%Then by calculation, we get $G(F)=K_{n-1}$. 
Note that $F$ corresponds to a zigzag PERT chart having $n$ activities with
the beginning (resp. ending) nodes being labelled odd (resp. even) and the activities are connected 
by an edge iff their labels differ by one.
\end{example}

Let $N(F)$ be the 1-skeleton of the Newton subdivision of $F$ (\cite{IMS}). 
Then $G(F)$ is a subgraph of $N(F)$ because the Newton subdivision of $F$ is dual to $V(F)$. 
Note that the Newton subdivision of $F$ equals the Newton polytope of $F$ 
since all the coefficients of $F$ are the same. 
%since all the coefficients of $F$ are units. 

%
%The following corollary of Corollary \ref{existence of C_I prop} holds.
%
%\begin{corollary}
%The number of vertices of $G(F)$ equals $\# {\mathcal I}$.
%\end{corollary}

Since the number of the vertices of $G(F)$ equals $\# \mathcal I$, 
the number of vertices of $G(F)$ and $N(F)$ coincides.
However, some edges in $N(F)$ may not appear in $G(F)$.

\begin{example}
Let $F$ be a prerealisable tropical polynomial
$$ F=t_1t_3t_6 \oplus t_1t_4t_5 \oplus t_1t_4t_6 \oplus t_2t_3t_6 \oplus t_2t_5.$$
Then $F$ is realisable, 
since there exists a corresponding PERT chart drawn left side of Figure 2. 
The right figure is $G(F)$, 
where boxes represents vertices with $I$ written inside, 
and the the bold lines are the edges (which was obtained by calculating $V(F)$). 
However, $N(F)$ contains the dotted lines too. 
In fact, $G(F)$ cannot be obtained as the 1-skeleton of a Newton subdivision of any prerealisable polynomial. 
This is because for any polytope of dimension $d$, the degree of each vertex must be bigger than $d$. 
Thus if there exists a prerealisable polynomial $F'$ satisfying $N(F')=G(F)$, 
$N(F')$ must be a polygon since there exists a vertex of degree 2. 
Then every vertex of $N(F')$ must be of degree 2, 
whereas $G(F)$ has a vertex of degree 3. 

Now we explain an important feature of adjacency. 
Let $\{1,4,6\}$ be the current critical path. 
Suppose we decrease the time cost of the activity $1$. 
Then the critical path may change. 
The new one should be $\{2,3,6\}$, 
because that is the only path which is adjacent to $\{1,4,6\}$ and does not contain $1$ as seen in the adjacency graph. 
%Since 1 is the only activity which changed its time cost, 
%the only possibility of the new critical path is $\{2,3,6\}$ from the adjacency graph. 
Or suppose the time cost of the activity $5$ has suddenly risen up. 
Then $\{1,4,5\}$, not $\{2,5\}$, is the only possibility. 
\end{example}

\begin{figure}[t]
%\begin{center}
\includegraphics{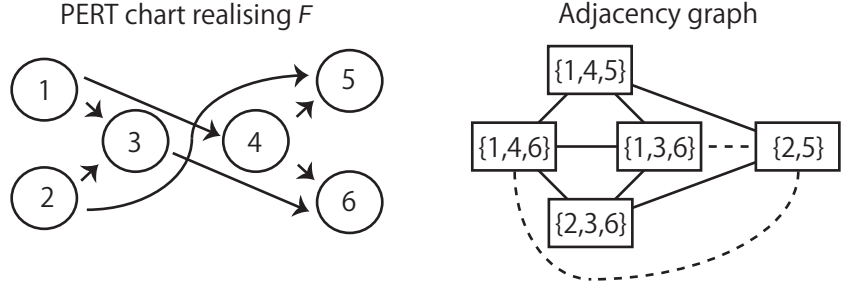}
\caption{$N(F)\neq G(F)$}
%\end{center}
\end{figure}

One should study $V(F)$ to obtain $G(F)$, which 
%can be done before the system starts, but 
is not very easy in general. 
When the PERT chart has some structure, one might obtain $G(F)$ accordingly. 
%In general, obtaining $G(F)$ is quite difficult since it requires calculating a $n-1$-dimensional tropical hypersurface $V(F)$. 
%In the following, we present cases that we can acquire $G(F)$ more easily. 
We shall discuss three cases: independency, product structure and homogeneity. 

First, an independent activity increases the connectivity of the adjacency graph. 
\begin{proposition}
If a term $t_I$ of $F$ contains a variable $t_i$ which does not appear 
in any other terms, then the vertex $I$ of $G(F)$ is connected to
all other vertices.
\end{proposition}
\begin{proof}
%Suppose $t_i$ is the variable that only appears in the term $t_I$.
Let $t_J$ be another term of $F$. 
%Then as we have seen before,
$\mbi{e}_J$ is an interior point of $C_J$. 
Take the point $\mbi{t}=\mbi{e}_J + s\mbi{e}_i$ for $s \geq 0$. 
When $s$ increases from $0$, 
$t_I$ is the only term in $F$ whose value at $\mbi{t}$ gets larger. 
Eventually we get $t_I=t_J$. 
%By increasing the value of the $k$-th coordinate
%of $\mbi{e}_J$, we can take a point $p$ such that terms $t_I$ and $t_J$ takes
%the same maximal value of $F$. 
%From the construction, these two terms are the only terms that attain the maximal value. 
Those $t_I,t_J$ are the only terms that attain the maximum at that $\mbi{t}$. 
Thus $C_I$ and $C_J$ are adjacent.
\end{proof}

Next we treat the product. 
For two undirected graphs $G_1=(V_1, E_1), G_2=(V_2, E_2)$, we define the Cartesian product $G_1 \square G_2$ as in \cite{IK}, 
that is:
\begin{enumerate}
 \item the vertex set being $V_1 \times V_2$,
 \item $\{(v_1, v_2), \ (v'_1, v'_2)\}$ is an edge of $G_1 \square G_2$ iff
       $$\{v_1, v'_1\} \in E_1, \ v_2=v'_2 \ \ \ \text{or} \ \ \ v_1=v'_1, \ \{v_2, v'_2 \} \in E_2.$$
\end{enumerate}

\begin{proposition}
Let $F_1, F_2$ be polynomials with no common variables. Then 
$$G(F_1F_2)=G(F_1) \square G(F_2).$$
\end{proposition}

\begin{proof}
Let the number of variables of $F_1, F_2$ be $n, m$, respectively. Since
$$V(F_1 F_2)_{\ge 0}=\bigl( V(F_1)_{\ge 0} \times \R_{\ge 0}^m \bigr) \cup
                         \bigl( \R_{\ge 0}^n \times V(F_2)_{\ge 0} \bigr) $$
holds, we have the proposition.
\end{proof}

\begin{remark}
If $F_1, \ F_2$ are realisable, then one can also realise $F_1 F_2$
by connecting every ending node of the PERT chart of $F_1$ to all the starting nodes of
the PERT chart of $F_2$.
\end{remark}

The Cartesian product of $n$ copies of $K_2$ is called the {\it $n$-hypercube graph\/} and denoted by $Q_n$ (\cite{IK}).

\begin{corollary}
$Q_n$ is the adjacency graph of the realisable tropical polynomial
$$F=\prod_{k=1}^n (t_{i_k}\oplus t_{j_k}).$$
\end{corollary}
\begin{proof}
%It is obvious that $Q_n$ is the Cartesian product of $n$ copies of
%a complete graph on two vertices $(K_2)^{\square n}$.
By Example \ref{complete}, $K_2$ is the adjacency graph of $F_{2,1}$. 
Therefore, $Q_n = (K_2)^{\square n}$ is that of the $n$ product of $F_{2,1}$ with independent variables. 
%Thus we obtain the corollary.
\end{proof}

Finally we treat the homogeneity. 
\begin{proposition}\label{G-N}
If $F$ is homogeneous, then
$G(F)=N(F)$.
\end{proposition}

\begin{proof}
Put $\deg F=d$, and fix a point $\mbi{t}$. Then the value of each term
of $F$ at $\mbi{t}+(a, \dots, a)$ is exactly $a \cdot d$ bigger than the value at $\mbi{t}$,
so if two chambers are adjacent, they are also adjacent within $\varGamma$.
\end{proof}

For a prerealisable tropical polynomial $F=\sum_{I \in \mathcal I} t_I$, 
put $F^{\vee}=\sum_{I \in \mathcal I} t_{[n] \setminus I}$. 

\begin{lemma}
$F^{\vee}$ is prerealisable.
\end{lemma}

\begin{proof}
%We will show that $F^{\vee}$ is nondivisble.
If $[n] \setminus I \subset [n] \setminus J$ holds, 
$I \supset J$ also holds and the nondivisibility condition of $F$ yields $I=J$. 
Thus we have $[n] \setminus I = [n] \setminus J$. 
\end{proof}

\begin{remark}
$F^{\vee}$ may not be realisable even if $F$ is realisable. 
%The simplest example of such is $F_{3,1}$. 
%We have $F_{1,3}^{\vee}=t_2t_3 \oplus t_1t_3 \oplus t_1t_2=F_{2,3}$.
For instance, 
Corollary \ref{realisable} show that $F_{3,1}$ is realisable, whereas $F_{3,1}^{\vee}=F_{3,2}$ is not.
\end{remark}

%$F^{\vee}$ is also prerealisable because if $[n] \setminus I \subset [n] \setminus J$ holds, $I \supset J$ also holds. 
%The nondivisibility condition on $F$ yields $I=J$, and we have $[n] \setminus I = [n] \setminus J$.

\begin{proposition}
If $F$ is homogeneous, then $G(F)=G(F^{\vee})$.
\end{proposition}

\begin{proof}
%In what follows, tropical expressions will be embraced by the quotation marks to avoid confusion. 
%In what follows, we assume all the expressions to be tropical. 
%For instance, $t_I^{-1}$ means $-t_I$ non-tropically.
For a point $p$, denote the point symmetric about $\mbi{0}=(0,\dots,0)$ as $-p$. 
Also denote the function taking the inverse value of a tropical monomial function $t_I$ with respect
to the tropical multiplication (usual addition) as $t_I^{-1}$. 

Since $t_I$ defines a linear function passing through $\mbi{0}$, its graph is symmetric about $\mbi{0}$.
% the value of $t_I$ at $-p$
%is the inverse of the value at $p$ with respect to the usual addition (tropical multiplication).
Thus the value of $t_I^{-1}$ at $-p$ and the value of $t_I$ at $p$ are the same.

Tropically adding these monomials, we obtain $F=\sum_I t_I$ and $F'=\sum_I t_I^{-1}$.
The value of $F$ at $p$ coincides to the value of $F'$ at $-p$, 
so the corner locus of $F$ and the corner locus of $F'$ are symmetric about $\mbi{0}$. 
By tropically multiplying $t_{[n]}$ to $F'$, we have $F^{\vee}$ because $t_I^{-1} \otimes t_{[n]}=t_{[n] \setminus I}$. 
Since tropically multiplying a monomial does not change the corner locus, 
$V(F)$ and $V(F^{\vee})$ are symmetric about $\mbi{0}$. 
Then we have $N(F)=N(F^{\vee})$ because of the duality between the Newton subdivision and the tropical hypersurface, 
%Newton subdivision is dual to the tropical hypersurface, 
and the proposition follows from Proposition \ref{G-N}.
\end{proof}

Since $F_{n,k}^{\vee}=F_{n,n-k}$ holds, we have the following corollary.

\begin{corollary}
$G(F_{n,k})$ and $G(F_{n,n-k})$ coincides.
\end{corollary}

\begin{example}
From the previous corollary and Example \ref{complete}, we obtain
$$G(F_{n,n-1})=G(F_{n,1})=K_n.$$
\end{example}

%\section{PERTに対応するTropical多項式の同値関係}
%
%$P_1=(V_1=[n], \ \mbi{e}_1,\ F_1),\ P_2=(V_2=[m], \ \mbi{e}_2,\ F_2)$をPERTとする。
%$P_1$から$P_2$への写像とは$\R_{\geq 0}$-加群の準同型
%$$ \phi: \ \R_{\geq 0}^n \longrightarrow \R_{\geq 0}^m$$
%であって、関数として$F_1=F_2 \circ \phi$を満たすものである。
%もし$\phi$に加え$P_2$から$P_1$への写像$\psi$が存在する場合、$P_1$と$P_2$は同値であるという。
%これは同値関係になる。
%
%\begin{example}
%道グラフ$P_n$と一点は同値である。
%$\phi(x_1,\ \cdots,\ x_n) = x_1 + \dots + x_n, \ \psi(x)=(x,\ 0, \cdots, 0)$
%とすればよい。
%\end{example}
%
%この同値関係により無駄な頂点をつぶせる？
%
%\begin{example}
%$F_1=(\Sigma_{i=1}^{n}t_{1i})t_2(\Sigma_{j=1}^{m}t_{3j})$と$F_2=(\Sigma_{i=1}^{n}t_{1i})(\Sigma_{j=1}^{m}t_{2j})$
%は同値である。
%$$\phi(x_{11},\ \dots,\ x_{1n},\ x_2,\ x_{31},\ \dots,\ x_{3m})=(x_{11}+x_2,\ \dots,\ x_{1n}+x_2,\ x_{31},\ \do%ts,\ x_{3m})$$
%$$\psi(x_{11},\ \dots,\ x_{1n},\ x_{21},\ \dots,\ x_{2m})=(x_{11},\ \dots,\ x_{1n},\ 0,\ x_{21},\ \dots,\ x_{2m%})$$
%とすればよい。
%\end{example}
%
%\begin{proposition}
%$\phi(V(F_1)) \subset V(F_2)$
%\end{proposition}
%
%\begin{proof}
%$\phi$は準同型なので写像$F_1$でsmoothでない点は$\phi$で送った先でもsmoothでない。すなわち$V(F_2)$の元である。
%\end{proof}
%
%逆は？すなわち$\phi(\R_{\geq 0}^n \setminus V(F_1)) \cap V(F_2) = \emptyset$?
%
%
%

\bibliographystyle{amsplain}

\end{document}